\documentclass[A4,12pt]{svjour3}
\usepackage{amssymb,amsmath,latexsym,color,
amsfonts}
\usepackage{hyperref}
\usepackage{mathrsfs}
\newcommand{\mh}{\mathcal{H}}
\newcommand{\mv}{\mathcal{V}}
\newcommand{\mw}{\mathcal{W}}

\newcommand{\real}{\mathbb R}

\newcommand{\Dt}{\Delta}
\newcommand{\norm}[1]{\left\Vert#1\right\Vert}
\newcommand{\abs}[1]{\left\vert#1\right\vert}
\newcommand{\rot}{\mathrm{rot}}
\newcommand{\nat}{I\!\! N}
\newcommand{\eps}{\varepsilon}

\newcommand{\dv}{\mathrm{div}}

\newcommand{\vect}{\mathrm{span}}

\newcommand{\ds}{\displaystyle}

\numberwithin{equation}{section} 

\DeclareMathOperator*{\esssup}{ess\,sup}
\DeclareMathOperator*{\essinf}{ess\,inf}

\usepackage{hyperref}

\newcommand{\la}{\langle}
\newcommand{\lf}{\left\langle\!\!\!\left\langle}
\newcommand{\rr}{\right\rangle\!\!\!\right\rangle}
\newcommand{\ra}{\rangle}

\begin{document}
\title{Weak solutions and optimal control of hemivariational evolutionary Navier-Stokes equations under Rauch condition}
\titlerunning{Hemivariational Inequality for Navier-Stokes Equations}        

\author{H. Mahdioui \and  S. Ben Aadi \and K. Akhlil  
}


\institute{ H. Mahdioui, ENSA, Ibn Zohr University, h.mahdioui@uiz.ac.ma\\
S. Ben Aadi, Ibn Zohr University, sultana.benaadi@edu.uiz.ac.ma\\
K. Akhlil, FPO, Ibn Zohr University, akhlil.khalid@gmail.com.
}

\date{Received: date / Accepted: date}
\maketitle

\begin{abstract}
In this paper, we consider evolutionary Navier-Stokes equations subject to the nonslip boundary condition together with a Clarke subdifferential relation between the dynamic pressure and the normal component of the velocity. Under Rauch condition, we use the Galerkin approximation method and a weak precompactness criterion to ensure the convergence to a desired solution. Moreover, a control problem associated with such system of equations is studied with the help of a stability result with respect to the external forces. In the end of this paper, a more general condition due to Z. Naniewicz, namely the directional growth condition, is considered and all the results are reexamined.

\end{abstract}

\subclass{35Q30 \and 47H10 \and 49J20 \and  49J52  \and 49J53}

\keywords{Hemivariational inequality, Navier-Stokes equations, nonsmooth optimization, Rauch condition, directional growth condition.}

\tableofcontents

\section{Introduction}

    In many engineering situations, one deals with fluids flow problems in tubes or channels, or for semipermeable walls and membranes. In practice, hydraulic control devices are used as a mechanism allowing the adjustment of orifices dimensions so that the normal velocity on the boundary of the tube is regulated to reduce the dynamic pressure. The model that describes usually this situation is repesented by Navier-Stokes equations for incompressible viscous fluids with the nonslip boundary conditions together with a Clarke subdifferential relation between the dynamic pressure and the normal component of the velocity. The resulting multivalued subdifferential boundary condition leads, after a standard variational transformation, to the so-called hemivariational inequality. 

The theory of hemivariational inequalities was introduced for the first time by Panagiotopoulos \cite{Pana81,Pana83,Pana89,Pana88,Panagiotopoulos1} for the sake of generalization of the classical convex variational theory to a nonconvex one. The main tool in this effort is the generalized gradient of Clarke and Rockafellar \cite{Clar75,Clar83,Rock80}. From this perspective, the litterature have seen a fast emergence of applications in a mathematical and mechanical point of view, see  \cite{Pana89,Pana88,Pana91,Pana882,Panasta88,Panakol90,Panaban84} for more details. Among the main applications of this theory, we mention Newtonian and non-Newtonian Navier-Stokes equations and their variants (Oseen model, heat-conducting fluids, miscible liquids,\dots) with nonstandard boundary conditions ensuring from multivalued nonmonotone friction law with leak, slip or nonslip conditions. For recent directions on the hemivariational theory we refer to \cite{LL,LT,XZS,migdud18}.

Over the last two decades intensive research has been conducted on hemivariational inequalities for stationary and non-stationary Navier-Stokes equations. For convex functionals, the problem has been studied essentially by Chebotarev \cite{chebo92,chebo97,chebo03}. We mention also \cite{chebo01} for stationary Boussinesq equations and \cite{kono00} by Konovalova for non-stationary Boussinesq equations. In all these papers the considered problems was formulated as variational inequalities. In the nonconvex case, the stationary case was considered by Migorski and Ochal \cite{M05HE} Migroski \cite{mig04}, for nonstationary case Migorski and Ochal \cite{M07HE}, see also \cite{fang16}. For an equilibrium approach one can see for example \cite{Benaadi}. On the other hand, optimal control problem involving hemivariational inequalities attracts more and more attention of researchers in recent years, we refer to the introductions of \cite{PGM20} and \cite{FH20} for a short review on the subject.

There are two main conditions that one can impose on the locally Lipschitz function under subdifferential effect, namely the classical growth condition or the Rauch condition due to J. Rauch \cite{rauch1977discontinuous}. The last one is less popular even if it was the main assumption in the beginning of the theory of hemivariational inequalities. The Rauch condition expresses actually the ultimate increase of the graph of a certain locally bounded function and is, in fact a special case of another unpopular condition, namely the directional growth condition due to Z. Naniewicz \cite{nani94}. As an advantage of the Rauch condition is that it allows to avoid smallness conditions(i.e. the relationship between the constants of the problem) brought by the classical growth condition. In the case of Navier-Stokes equations, the smallness condition links the growth condition constant, the coercivity constant and the norm of the trace operator. It is however not clear how it can be checked in a concrete situation. Another advantage, is that it allows us to consider "stanger" functions at infinity. In fact the only thing we require to the function is that the essential supremum of the function on the left side to be  greater than the essential infimum on the right side.

Among the disadvantages of Rauch condition is that although it ensures the existence of a solution, does not allow the conclusion that the non-convex functional is
locally Lipschitz or even finite on the whole space. The Aubin-Clarke formula can't be used and a slight change on the definition of a solution have to be made. On the other hand we are looking for the dynamical pressure in a larger space which makes harder the question of uniqueness without a classical growth condition even if a monotonicity type assumption is acquired\cite{MBA}. Finally it is worth to mention that there is no direct link between the Rauch condition and the classical growth condition and the choice depends mainly on the concrete situation.

The present paper represents a continuation of our previous paper \cite{MBA} where existence and optimal control questions involving stationary Navier-Stokes problem with multivalued nonmonotone boundary condition are studied. In this paper we tackle the non-stationary problem. Always under Rauch condition, we use the Faedo-Galerkin approximation to regularize the system at the level of the multivalued boundary condition and we use the fact that the approximation sequence so-obtained is weakly precompact in the space of integrable functions. We also take advantage of the techniques used in \cite{M07HE} at the level of the nonlinear term to ensure the convergence of the approximate sequence to the desired solution. This study can be also done with the directional growth condition as a generalization. The question of the existence of an optimal control is important in applications. We tackle this subject in the spirit of the works of Barbu \cite{Ba} and Migorski \cite{law2013note}.

The outline of this paper is as follows. In section 2 we state the problem and give its hemivariational form by using the Lamb formulation. In section 3 we regularize our problem by using the Faedo-Galerkin approximation method and prove the existence of solutions to the regularized problem. By combining techniques from \cite{M07HE} and \cite{MBA} we will provide an existence result in section 4. Section 5 is devoted to the optimal control problem subjected to our evolutionary hemivariational inequality, while section 6 is dedicated to the directional growth condition as a generalization of the Rauch condition.

\section{Problem statement}
\label{sec:2}
Let $\Omega$ be a bounded simply connected domain in $\real^d$ with $d=2,3 $ with connected boundary $\partial\Omega$ of class $C^2$ and $\Omega_T= (0,T) \times \Omega $ where $T>0$. We consider the following evolution Navier-Stokes system:
\begin{align}\label{e}
u' -\nu \Dt u + (u.\nabla u) + \nabla p  = f_{\text{ext}},  \text{ in } \Omega_T\\
 \dv u = 0 \,
 \text{ in } \Omega_T\\
u(0)=u_0 \text{ in } \Omega,
\end{align}
This system describes the flux of an incompressible viscous fluid in a domain $\Omega$ subjected to an external force $f_{\text{ext}} = \{f_{\text{ext},k}\}_{k=1}^d$.  $u = \{u_k\}_{k=1}^d$, $p$ and $\nu$ denote respectively the velocity, the pressure and the kinematic viscosity of the fluid. The nonlinear term $(u.\nabla)u$ (called the convective term) is the symbolic notation of the vector $\sum_{j=1}^d u_j \frac{\partial u_i}{\partial x_j}$. As usual we use the Lamb formulation \cite[Chapter I]{Gir86} to rewrite the evolution Navier-Stokes system as follows:
\begin{equation}\label{eq02}
u'+ \nu \rot\,\rot\, u + \rot\, u\times u + \nabla \widetilde p  = f_{\text{ext}}, \text{ in } \Omega_T
\end{equation}
\begin{equation}\label{eq3}
 \quad \dv u = 0 \,
  \mbox{ in } \Omega_T .
\end{equation}
\begin{equation}
u(0)=u_0 \, \mbox{ in } \Omega.
\end{equation}\label{eqq }
where $\widetilde p=p+\frac{1}{2}\abs{u}^2$ is the total head of the fluid, or "total pressure" .

We suppose that on the boundary $\partial\Omega$ the tangential component of the velocity vector are known, and
without loss of generality we put them equal to zero (the nonslip condition):
\begin{equation}
u_\tau(t,x):=u(t,x)-u_N(t,x)n \, \text{ on } \partial\Omega_T:=(0,T)\times \partial\Omega 
\end{equation}
where $n=\{n_k\}^d_{k=1}$ is the unit outward normal on the boundary $\partial\Omega$ and $u_N(t,x)=u(t,x).n=\sum_{1}^{d}u_i(t,x)n_i$ denotes the normal component of the vector $u$. Moreover, we assume the following subdifferential boundary condition:
\begin{equation}\label{eq4}
\widetilde p(t,x) \in \partial j(u_N(t,x)) \text{ on } \partial\Omega_T
\end{equation}where $\partial j (\xi)$ is the Clarke subdifferential of   $j$ at $\xi$ and is given by

\begin{center}
$\partial j(\xi)= \{ \xi^*\in V^* : j^0(\xi;h) \geq \la \xi^*,h\ra_{V^*\times V} \text{ for all } h\in V \}$,
\end{center}
and $j^0(\xi;h)$ is the generalized derivative of a locally Lipschitz function $j$ at $\xi\in V$ in the direction $h\in V$ defined by:
\begin{center}
$j^0(\xi;v)=\ds\limsup_{\nu \rightarrow \xi, \lambda \downarrow
0} \frac{j(\nu+\lambda v)-j(\nu)}{\lambda }.$
\end{center}

To work conveniently on the problem \eqref{eq02}-\eqref{eq4}, we need the following functional spaces:
\begin{align*}
C=& \{ u \in \mathcal{C}^\infty(\Omega;\real^d) : \dv u=0 \text{ in } \Omega, u_\tau=0 \text{ on } \partial\Omega \}\\
V=&\text{ the closure of } C \text{ in the norm of } H^1(\Omega;\real^d)\\
H=&\text{ the closure of } C \text{ in the norm of } L^2(\Omega;\real^d)
\end{align*}
Then we have $ V \subset H \simeq H^* \subset V^*$, with all the embedding being continuous and compact. Moreover, 
for an interval time $[0,T]$, we introduce the following spaces:
\begin{align*}
\mv=& L^2(0,T; V)\\
\mh=& L^2(0,T; H)\\
 \mw=&\{ u\in \mv: u'\in \mv^* \}
\end{align*} 
Then, we have also the following continuous embedding, $ \mw \subset \mv \subset \mh \subset \mv^*.$

We consider the operators $\mathscr A:V\rightarrow V^*$ and $\mathscr B:V\times V\rightarrow V^*$ defined by:
\begin{equation}
\la \mathscr A u, v \ra= \nu \int_\Omega \rot\, u. \rot\, v \,{\rm d}x
\end{equation}
\begin{equation}
\la \mathscr B(u,v), w\ra=\int_\Omega (\rot\, u \times v). w \,{\rm d}x
\end{equation}
for all $u,v,w \in V$. As usual, we will use the notation $\mathscr B[.]=\mathscr B(.,.)$. It is well known (cf.\cite{{Byk60}}) that if the domain $\Omega$ is simply
connected, the bilinear form 
\begin{equation}
(\!(u,v)\!)_V = \int_\Omega \rot\, u. \rot\, v\, {\rm d}x
\end{equation}
 generates a norm in $V$, $\norm{u}_V= (\!(u,v)\!)_V^\frac{1}{2}$ which is equivalent to the $H^1(\Omega, \real^d)$-norm. Hence, it is clear 
that the operator $\mathscr A$ is coercive.

In order to give the weak formulation to the problem \eqref{eq02}-\eqref{eq4}, we multiply it by a certain $v \in V$ and apply the Green formula. We obtain:
\begin{equation}\label{eq6}
\la u'(t)+\mathscr Au(t)+\mathscr B[u(t)],v\ra+\ds\int_{\partial\Omega }\widetilde{p}(t,x)v_N {\rm d}\sigma(x)=\la f(t),v\ra,
\end{equation}where $\la f(t),v\ra=\int_\Omega f_{\text{ext}}(t).v\,dx$.  From the relation (\ref{eq4}), by using the definition of the Clarke subdifferential, we have
\begin{equation}\label{eq7}
\ds\int_{\partial\Omega} \widetilde{p}(t,x) v_N(x){\rm d}\sigma(x) \leq \int_{\partial\Omega} j^0(x,u_N(x);v_N(x)){\rm d}\sigma(x).
\end{equation}
The relations (\ref{eq6})-(\ref{eq7}) yields to the following weak formulation 
\begin{equation*}
\text{(EHVI)} \left\{ 
\begin{array}{ll}
\text{ Find } u\in \mathcal W \text{ such that for all } v\in V \text{ and a.e. }t \in (0,T)\\ \la u'(t)+ \mathscr Au(t)+\mathscr B[u(t)],v\ra+\ds\int_{\partial\Omega} j^0(t,x,u_N(t,x);v_N(x)){\rm d}\sigma(x) \geq \la f(t),v\ra,\\
u(0)=u_0.

\end{array}
\right.
\end{equation*}
The equation above is called an hemivariational inequality.

We have already mentioned in the introduction that the Rauch assumption is not sufficient to make the functional $J(u)=\int_{\partial\Omega} j(u)\,d\sigma$ locally lipschitz or even finite in the whole space $\mathcal V$. Because of this reason, a slight modified definition of being a solution should be adopted. Define the functional space
\[
L_N^\infty(\partial\Omega)=\{u,\,u_N=\gamma(u).n\in L^\infty(\partial\Omega;\mathbb R)\}\]
where $\gamma$ is the trace operator from $ V$ in $L^2(\partial\Omega;\mathbb R^d)$. Now, we are able to give what we mean by a solution to the problem $(EHVI)$.

\begin{definition}

A function $u\in \mathcal W$ is said to be solution of (EHVI) if there exists $\kappa\in  L^1((0,T)\times \partial \Omega,\,\real)$ such that: For a.e $ t\in (0,T)$
\begin{equation*}\label{eqdef}
\left\{ 
\begin{array}{ll}
\la u'(t)+\mathscr Au(t)+\mathscr  B[u(t)],v\ra+\ds\int_{\partial\Omega} \kappa(t,z)\,v_N(z)\, d\sigma(z) = \la f(t),v\ra, \forall v\in\ V\cap L_N^\infty(\partial\Omega)\\
\kappa (t,z)\in\partial j(t,z;u_N(t,z)),\,\text{for a.e.}\, (t,z) \in(0,T)\times\partial\Omega\\
u(0)=u_0.

\end{array}
\right.
\end{equation*}

\end{definition}
Note that since $ \mathcal W \subset C(0,T;H)$ continuously the initial condition $u(0)= u_0$  makes sense in $H$. To justify the above definition we refer to \cite{MBA} and \cite{nani94}.

\section{Regularized Problem}

 In what follows we restrict our study to superpotentials $j$ which are independent of $x$ and which subdifferential is obtained by "filling in the gaps" procedure (cf.\cite{rauch1977discontinuous}). Let $\theta \in L^\infty_{loc}(\real)$, for $\eps>0$ and $t\in \real$, we define:
\begin{equation*}
\underline{\theta}_\eps(t)=\displaystyle\essinf_{\abs{t-s}\leq\eps}\theta(s),\quad  \quad \overline{\theta}_\eps(t)=\displaystyle\esssup_{\abs{t-s}\leq\eps}\theta(s).
\end{equation*}
For a fixed $t\in \real$, the functions $\underline{\theta}_\eps,  \overline{\theta}_\eps$ are decreasing and increasing in $\eps$, respectively. Let 
\begin{align*}
\underline{\theta}(t)=\lim_{\eps \to 0^+} \underline{\theta}_\eps(t), \quad\quad \overline{\theta}(t)=\lim_{\eps \to 0^+}\overline{\theta}_\eps(t),
\end{align*}
and let $\hat{\theta}(t):\real \to 2^\real$ be a multifunction defined by
\begin{equation*}
\widehat{\theta}(t)=\left[ \underline{\theta}(t),\overline{\theta}(t) \right]
\end{equation*}
From Chang \cite{chang80} we know that a locally Lipschitz
function $j:\real \to \real$ can be determined up to an additive constant by the relation $$j(t)=\int_0^t \theta(s){\rm d} s $$ such that $\partial j(t) \subset \hat{\theta}(t)$ for all $t\in\real$. If moreover, the limits $\theta(t\pm 0)$ exist for every $t\in\real$, then $\partial j(t) = \hat{\theta}(t)$.\\
In order to define the regularized problem, we consider the mollifier \[\mathfrak h\in C_0^\infty(-1,1), \mathfrak h\geq 0 \text{ with } \ds\int_{-\infty}^{+\infty} \mathfrak h(s)\, ds=1
\] and let
\[\theta_\epsilon =\mathfrak h_\epsilon \star \theta
\text{ with } \mathfrak h_\epsilon(s)=\frac{1}{\epsilon}\mathfrak h(\frac{s}{\epsilon})
\]where $\star$ denotes the  convolution product. 

Consider the following auxiliary problem associated to (EHVI):
\begin{equation*}
(\mathscr P_\eps) \left\{ 
\begin{array}{ll}
\text{ Find } u \in  \mathcal W \text{ such that for all } v\in V\cap L_N^\infty(\partial\Omega), \text{ a.e. } t\in (0,T)\\
\ds \la u'(t)+ \mathscr Au(t)+ \mathscr B[u(t)],v\ra+\int_{\partial\Omega} \theta_\eps(u_N(t,x))v_N~ {\rm d}\sigma =\la f(t),v\ra \\
u(0)=u_0. \end{array}
\right.
\end{equation*}

Now and in order to define the corresponding finite dimensional problem, we  shall  use  the Faedo-Galerkin approximation approach. Let us consider a Galerkin basis $\{z_1,z_2,\dots\}$ in $V\cap L_N^\infty(\partial\Omega) $, i.e. $\{z_1,z_2,\dots\}$  forms at most a countable sequence of elements of $V\cap L_N^\infty(\partial\Omega)$, finitely $\{z_1,z_2,...,z_m\}$  are linearly independent. Consider $V_m=\vect\{z_1,z_2,\dots,z_m)$, we have $V_m \subset V_{m+1}$ and $ \ds\overline{\cup_{m \geq 0} V_m}=V\cap L_N^\infty(\partial\Omega)$. Moreover, the family $\{V_m\}_m$ satisfies
\[ \forall v \in V\cap L_N^\infty(\partial\Omega), \,\exists (v_m)_m,\, v_m \in V_m \text{ such that } v_m \to  v \text{ in } V\cap L_N^\infty(\partial\Omega), \text{ as  } n \to +\infty.\]
Let $\{u_m(0)\}$  be an approximation of the given  initial value $u_0$ such that $u_m(0) \in V_m$ for $ m\in \nat$ and suppose  that 
\begin{equation*}
u_{m0} \to u_0 \text{ in } H, \text{as } m \to +\infty.
\end{equation*} and,  

\begin{equation*}
(u_{m0})_m \text{ is bounded in } V\cap L_N^\infty(\partial\Omega).
\end{equation*}
We consider the following regularized Galerkin system of finite dimensional differential equations associated to (EHVI):
\begin{equation*}
(\mathscr P^m_\eps) \left\{ 
\begin{array}{ll}
\text{ Find } u_m \in \mathcal W_m \text{ such that for all } v\in V_m, \text{a.e. } t\in (0,T)\\
\ds \la u'_m(t)+\mathscr Au_m(t)+ \mathscr B[u_m(t)],v\ra+\int_{\partial{\Omega}} \theta_{\eps_m}(u_{mN}(t))v~ {\rm d}\sigma =\la f(t),v\ra \\
u_m(0)=u_{m0}.
\end{array}
\right.
\end{equation*}where $\mw_m=\{u\in L^2(0,T; V_m):\, u'\in L^2(0,T; V_m)\}$. 

The generalized derivative $\mathscr Lu = u'$ restricted to the subset $D(\mathscr L)= \{ u\in \mv: u'\in \mv^* \text{ and  } u(0)= u_0 \} = \{ u\in \mw : u(0)= u_0 \}$ defines a linear operator $\mathscr L: D(\mathscr L) \subset \mv \to \mv* $
given by
\[ \lf \mathscr Lu,v\rr= \int_0^T \la u'(t),v(t)\ra {\rm d}t \text{ for all } v\in \mv .\]

For the existence of solutions we will need the following hypothesis $H(\theta)\quad$ :

\begin{itemize}
\item[(1) ](Chang assumption) $\theta \in L_{loc}^{\infty}(\real),\, \theta(t \pm 0)\text{ exists for any } t \in \real.$
\item[(2) ] (Rauch assumption) there is  $\delta_{0}>0$ such that:
\[
\displaystyle\esssup_{ ]-\infty,-\delta_{0}[ } \theta(t) \leq 0 \leq \displaystyle\essinf _{ ]\delta_{0},+\infty[ } \theta(t)
\]
\end{itemize}

\begin{remark}
If one assume more generally that

\begin{equation}\label{eqdefalpha}
\displaystyle\esssup_{ ]-\infty,-\delta_{0}[ } \theta(t) \leq \alpha \leq \displaystyle\essinf _{ ]\delta_{0},+\infty[ } \theta(t)
\end{equation}
for some real number $\alpha$, it is possible to come back to the situation where the Rauch assumption is imposed by simply replacing $\theta$ by $\theta -\alpha$ and $f$ by $f-\alpha$. 
\end{remark}
\begin{remark}
We point out that the Rauch and the growth conditions are completely independent. Indeed, by taking examples, we show that neither of both conditions implies the other.  In fact, consider the function $\beta: \real \to \real$ defined by:
\begin{equation*}
\beta(t)=  \left \{  
\begin{array}{ll}
 \left\lfloor t^3\right\rfloor &\text{ if } \abs{t} \geq 1, \\
 -t &\text{ if } \abs{t} < 1.
  \end{array}
\right.
\end{equation*}
where $\left\lfloor t\right\rfloor$ stands for the integer part of $t$. One can prove easily (eventually by a contradiction argument), that  the function $\beta$ satisifies the Rauch condition while the growth condition can not be satisfied. Conversely, one can take a function $\beta: \real \to \real$ defined by $\beta(t)= 1+\sqrt{\abs{t}}$, it's clear that it satisfies the growth condition but not the Rauch condition as $\beta$ is positive for negative values.
\end{remark}

\begin{lemma}\label{lem1}
Suppose that $H(\theta)$ holds. Then we can determine $\rho_1,\rho_2 > 0$ such that for every $u\in V_m$
\begin{equation}\label{eq14}
\int_{(0,T)\times \partial{\Omega}} \theta_{\eps}(u_N(t,z))u_N(t,z) {\rm d}\sigma(z) {\rm d}t \geq -\rho_1\rho_2 T. \sigma(\partial{\Omega}).
\end{equation}
\end{lemma}
\begin{proof}
This is a classical result in the stationary case(cf. \cite[Lemma 3.2]{MBA}. It suffices to integrate over $t\in (0,T)$ to obtain the result.
\end{proof}

\begin{proposition}\label{pro2}
The sequence $(\theta_{\eps_m}(u_m))_{m\in\nat}$ is weakly precompact in $L^1((0,T)\times \partial{\Omega})$.
\end{proposition}
\begin{proof}
The proof is similar to \cite[Proposition 3.7]{MBA} with minor changes consisting mainly in replacing $\partial\Omega$ by $(0,T)\times \partial{\Omega}$ and remarking that $\la u_m'(t),u_m(t)\ra\geq 0$ for a.e $t\in[0,T]$.
 
\end{proof}

\begin{proposition}\label{prop2}
The regularized problem $(\mathscr P^m_\eps)$ has at leat one solution $u_m$.
\end{proposition}
\begin{proof}

We substitute $u_m(t)= \sum_{k=1}^m c_{km}(t)z_k$ in $(\mathscr P^m_\eps)$ to obtain

\begin{align}\label{eq}
\ds &\sum_{k=1}^m c_{km}'(t)\la z_k,z_i\ra+  \sum_{k=1}^m c_{km}(t)\la A z_k,z_i\ra+ \sum_{k,j=1}^m c_{km}(t)c_{jm}(t)b(z_k,z_j,z_i) \\ &+\int_{\partial\Omega} \theta_{\eps_m}(\sum_{k=1}^m c_{km}(t) z_k.n)z_i.n~ {\rm d}\sigma =\la f(t),z_i\ra, i=1,....m, \text{ a.e. } t\in [0,T] \nonumber\\
& c_{im}(0)=\alpha_{i0}, i=1,...,n.
\end{align}
The matrix with elements $\la z_k,z_i\ra, 1\leq i,k\leq m$ is nonsingular (i.e  $\det\{\la z_k,z_i\ra \}_{k,i= 1 }^m \neq 0$), we invert the matrix, then the equation \eqref{eq} can be written in the usual form:
\begin{align}\label{system}
c_{km}'(t)=\sum_{i=1}^m \beta_{ki}\la f(t),z_i\ra-\sum_{i=1}^m\alpha_{ki} c_{im}(t)- \sum_{i,j=1}\xi_{kij}c_{km}(t)c_{jm}(t) \\ -\sum_{i=1}^m \varrho_{ki}\int_{\partial\Omega} \theta_{\eps_m}(\sum_{k=1}^m c_{km}(t) z_k.\textbf{n})z_i.\textbf{n}~ {\rm d}\sigma \nonumber
\end{align}
where the initial values $c_{km}(0),\,k = 1,...,m$ are given, i.e.  $u_{m0}=\sum_{k=1}^m c_{km}(0) z_k$ 
\begin{equation}\label{initial}
c_{km}(0) \text{ is the } k^{\text{th}} \text{ component of  } u_{m0}.
\end{equation}
The differential system \eqref{system} with the initial condition \eqref{initial} define uniquely the scalar $c_{km}$ on the interval $[0,t_m)$. Then the solution $u_m$ exists on $[0,t_m)$, we can extended it on the closed interval $[0,T]$ by using a priori estimates in Lemma \ref{prop3}.
Since the scalar function $t \to \la f(t),z_i\ra $ in the equation \eqref{eq} are square integrable, so are the functions $c_{km}$ and therefore, for each $m$ we have:
\begin{equation}\label{eq15}
u_m \in L^2(0,T;V) \text{ and } u_m' \in L^2(0,T; V^*).
\end{equation}

\end{proof}

\section{Existence result}
 
 In this section we will prove the existence of solutions to the problem (EHVI) by analysing the convergence of the sequence $(u_m)_m$ solutions to $(\mathscr P^m_\eps)$. To do so we need some a priori estimates
 
\begin{lemma}\label{prop3}
The solution $\{u_m\}_m$ is bounded in $L^2(0,T;V)\cap L^\infty(0,T;H)$.
\end{lemma}
\begin{proof}
From Proposition \ref{pro2}, the regularized problem  $(\mathscr P^m_\eps)$ has at least one solution $\{u_m\}_m$. By replacing $v$ by $u_m(t)$ in $(\mathscr P^m_\eps)$, we get for a.e $t\in [0,T]$:
\begin{equation}\label{eq16}
\la  u'_m(t)+ \mathscr A u_m(t)+ \mathscr B [u_m(t)],u_m(t)\ra  
+\int_{ \partial \Omega}\theta_{\eps_m}(u_{mN}(t,x))u_{mN}(t,x) {\rm d}\sigma = \la f(t), u_m(t) \ra
\end{equation}
Because of \eqref{eq15} we have:
\begin{equation}
 \la  u'_m(t),u_m(t) \ra = \frac{1}{2}\ds\frac{\rm d}{\rm dt}\abs{u_m(t)}^2
\end{equation}
Then the equation \eqref{eq16} becomes:
\begin{align}\label{eq17}
\frac{d}{dt}\abs{u_m(t)}^2&+  2 \la \mathscr A u_m(t), u_m(t) \ra +2 \int_{\partial\Omega}\theta_{\eps_m}(u_{mN}(t,x))u_{mN}(t,x) {\rm d}\sigma \\ &
=2 \la f(t), u_m(t) \ra_{V^*} \text{ for a.e. } t\in[0,T]\nonumber
\end{align}
By the coerciveness of $\mathscr A$, Cauchy-Schwartz inequality and Young inequality we obtain
\[
\frac{d}{dt} \abs{u_m(t)}^2+M\norm{u_m(t)}^2+2 \int_{\partial\Omega}\theta_{\eps_m}(u_{mN}(t,x))u_{mN}(t,x) {\rm d}\sigma\leq \frac{2}{M}\norm{f(t)}_{V^*}
\]for a.e. $t\in (0,T)$( $M$ is the constant of coercivity). Integrating the previous equation from 0 to $s$, $0\leq s\leq T$ and using Lemma \ref{lem1}, one has: 

\begin{equation}\label{18}
\abs{u_m(s)}^2 +M\int_0^s\norm{u_m(\tau)}_V^2{\rm d}\tau \leq \frac{2}{M}\int_0^s \norm{f(\tau)}^2_{V^*}\,{\rm d}\tau+2\rho_1\rho_2 s. \sigma(\partial{\Omega}) +\abs{u_m(0)}^2
\end{equation}
Hence
\begin{align*}
\sup_{s\in[0,T]}\abs{u_m(s)}^2 \leq  \frac{2}{M} \norm{f}^2_{\mv^*}+2\rho_1\rho_2 T. \sigma(\partial{\Omega}) +\abs{u_m(0)}^2,
\end{align*}
The right hand side of the previous inequality, is finite and independent of $m$. We deduce that $\{u_m\}_m$ is bounded in $L^\infty(0,T;H)$.

Again from \eqref{18} we have:

\begin{equation*}
 M \int_0^T \norm{u_m(t)}_V^2{\rm d}t \leq \frac{2}{M} \norm{f}^2_{\mv^*}+2\rho_1\rho_2 T. \sigma(\partial{\Omega}) + \abs{u_m(0)}^2.
\end{equation*}
Then:
\begin{equation}\label{25}
 \norm{u_m}^2_{\mv} \leq \frac{2}{M^2} \norm{f}^2_{\mv^*}+\frac{2}{M}\rho_1\rho_2 T. \sigma(\partial{\Omega}) + \frac{1}{M}\abs{u_m(0)}^2
\end{equation}
then $\{u_m\}_m$  remains in a bounded subset of $\mv$. 

\end{proof}

\begin{theorem}\label{existence}
Under assumption $H(\theta)$, the problem (EHVI) has at least one solution.
\end{theorem}
\begin{proof} From Proposition \ref{pro2} and Proposition \ref{prop3}, we get
\begin{align}
u_m &\longrightarrow u \text{ weakly in } \mv \\
u_m& \longrightarrow u \text{ weakly-star  in } L^\infty(0,T;H)\\
\theta_{\eps_m}(u_{mN})& \longrightarrow \kappa \text{ weakly in } L^1((0,T)\times \partial\Omega)
\end{align}

Now, we focus on the weak convergence of the nonlinear term $\mathscr B[u_m]$ by using exactly the same procedure as in \cite{M07HE}. For the case $d=2$, we obtain from Temam \cite{Temam} that 
 \[ \norm{\mathscr B[u_m]}_{\mv^*} \leq c\norm{u_m}_{L^\infty(0,T;H)}\norm{u_m}_{\mv} \text{ with } c>0. \]
Moreover, the operator $\mathscr A$ is continuous. Hence $\{u'_m\}$ is
bounded in $\mv^*$. Thus, by passing to a next subsequence, if necessary, it follows
 \[  u'_m \longrightarrow u' \text{ weakly in } \mw.\]
Using the facts that $\mw \subset C(0,T;H)$ continuously, $\mw \subset \mh$ compactly and $\mw \subset L^2(0,T;L^2(\Gamma, \real^n))$ compactly, we have $u \in C(0,T;H)$ and
\[ u_m \to  u \text{ in } \mh,\,\, \gamma(u_m) \to \gamma(u) \text{ in } L^2(0,T;L^2(\Gamma;\real^d)).\]
Since $ u_m \to u$ weakly in $\mv$ and in $\mh$, analogously as in Ahmed \cite{ahmed92}, we have $B[u_m] \to
B[u]$ weakly in $\mv^*$. We remark that if $d=3$ we also have the convergence of $B[u_m] \to B[u]$ weakly in $\mv^*$ by a compactness embedding theorem as in \cite{ahmed92}.

Let $ \phi\in C_0^\infty(0,T)$ and $v \in V\cap L^\infty_N(\partial\Omega)$ . Then, there exists $\{v_m\}_{m\in\mathbb N}$ such that $v_m \in V_m$ and  $v_m \to v$ in $\mv$, as $ m \to \infty$. Denoting $\psi_m(x,t)= \phi(t) v_m(x)$ and $\psi (x,t)= \phi(t) v(x)$, we have $\psi_n \to \psi $ in $\mw$. From $(\mathscr P_\varepsilon^m)$, it follows 

\begin{align}\label{20} 
\int_0^T \la u'_m(t)+&\mathscr Au_m(t)+\mathscr B[u_m(t)],\psi_m(t)\ra\\
+& \int_{(0,T)\times \partial\Omega}\theta_{\eps_m}(u_{mN}(t))\psi_{Nm}(t) {\rm d}\sigma {\rm d}t=\int_0^T \la f(t),\psi_m(t)\ra\, dt .
\end{align}
Using the above convergences, letting $ m\to +\infty$, we obtain:
\[ \int_0^T \la u'(t)+\mathscr Au(t)+\mathscr B[u(t)],v\ra \phi(t) {\rm d}t +\int_{(0,T)\times\partial\Omega} \kappa
. v_N  \phi (t)  { \rm d}t\,d\sigma=\int_0^T \la f(t),v\ra \phi(t) { \rm d}t.\]
Since $\phi$ is arbitrary, we deduce that
\[ \la u'(t)+\mathscr A u(t)+\mathscr B[u(t)],v\ra+ \int_{\partial\Omega} \kappa
. v_N   { \rm d}\sigma= \la f(t),v\ra\]
for a.e. $t \in (0,T)$ and for all $v\in V\cap L^\infty_N(\partial\Omega)$. 
 
In order to complete the proof it will be shown that
\begin{equation*}
\kappa(t,z) \in \widehat{\theta}(u_N(t,z))=\ds \partial^c j(u_N(t,z)),\quad\text{ for a.e }  (t,z)\in [0,T]\times\partial{\Omega}.
\end{equation*}
Since $\gamma(u_m)\to\gamma(u)$ in $L^2(0,T;L^2(\partial\Omega))$, we obtain $u_{mN}\to u_N$ in $L^2(0,T;L^2(\partial\Omega))$ and consequently $u_{mN}(t,x)\to u_N(t,x)$ for a.e. $(t,x)\in [0,T]\times\partial\Omega$, then by applying Egoroff's theorem we can find that for any $\alpha > 0$ we can determine $\omega\subset [0,T]\times\partial{\Omega}$   with  $\sigma(\omega) <\alpha$  such that
\begin{equation*}
u_{mN}\rightarrow u_N,\quad\text{ uniformly on } [0,T]\times\partial{\Omega} \setminus \omega,
\end{equation*} with $u_N\in \mathrm L^{\infty}([0,T]\times\partial{\Omega} \setminus \omega)$. Thus for any $\alpha > 0$ we can find $\omega \subset [0,T]\times\partial{\Omega}$ with $\sigma(\omega) <\alpha$
such that for any $\mu> 0$ and for $\varepsilon < \varepsilon_0 < \mu/2$ and $n > n_0 > 2/\mu$ we have
\begin{equation*}
|u_{mN}- u_N|<\frac{\mu}{2},\quad\text{ on } [0,T]\times\partial{\Omega}\setminus\omega.
\end{equation*}
Consequently, one obtain that

\begin{equation*}
\begin{array}{ll}
\theta_{\eps}(u_{mN})&\leq \displaystyle\esssup_{|u_{mN}-  \xi|\leq \varepsilon} \theta(\xi)\\
&\leq \displaystyle\esssup_{|u_{mN}-  \xi|\leq  \frac{\mu}{2}} \theta(\xi)\\
&\leq \displaystyle\esssup_{|u_N-  \xi|\leq  \mu} \theta(\xi)\\
&= \overline{\theta}_\mu(u_N)
\end{array}
\end{equation*}
Analogously we prove the inequality

\begin{equation*}
 \underline{\theta}_\mu(u_N)=\displaystyle\essinf_{|u_N -\xi|\leq \mu}\theta(\xi)\leq \theta_\eps(u_{mN})
\end{equation*}

We take now $v \geq 0$ a.e. on $[0,T]\times\partial{\Omega}\setminus\omega$  with $v\in \mathrm L^\infty([0,T]\times\partial{\Omega}\setminus\omega ) $. This implies 

\begin{equation*}
\int_{ [0,T]\times\partial{\Omega}\setminus\omega} \underline{\theta}_\mu(u_N)\,v {\rm d}\sigma\leq \int_{[0,T]\times \partial{\Omega}\setminus\omega} \theta_\eps(u_{mN} )\,v {\rm d}\sigma\leq 
\int_{ [0,T]\times\partial{\Omega}\setminus\omega}\overline{\theta}_\mu(u_N)\,v {\rm d}\sigma
\end{equation*}
Taking the limits as $\eps \rightarrow 0$ and $m\rightarrow \infty$ we obtain that
\begin{equation*}
\int_{ [0,T]\times\partial{\Omega}\setminus\omega} \underline{\theta}_\mu(u_N)\,v {\rm d}\sigma\leq \int_{[0,T]\times \partial{\Omega}\setminus\omega} \kappa\, v {\rm d}\sigma\leq\int_{ \partial{[0,T]\times\Omega}\setminus\omega}\overline{\theta}_\mu(u_N)\,v {\rm d}\sigma
\end{equation*}
and as $\mu\rightarrow 0^+$  that
\begin{equation*}
\int_{ [0,T]\times\partial{\Omega}\setminus\omega} \underline{\theta}(u_N)\, v {\rm d}\sigma\leq \int_{[0,T]\times \partial{\Omega}\setminus\omega} \kappa\, v {\rm d}\sigma\leq\int_{[0,T]\times \partial{\Omega}\setminus \omega}\overline{\theta}(u_N)\,v {\rm d}\sigma
\end{equation*}

Since $v$ is arbitrary we have that

\begin{equation*}
\kappa\in [ \underline{\theta}(u_N), \overline{\theta}(u_N)     ]=  \widehat{\theta}(u_N)
\end{equation*}

where $\sigma(\partial\Omega) < \alpha $. For  $\alpha $ as small as possible, we obtain the result.
\end{proof}

\section{Optimal Control}

In this section, we provide a result on dependence of solutions with respect to the density of the external forces and use it to study the distributed parameter optimal control problem corresponding to it.

Let $f\in L^2(0,T;V^*)$. Under $H(\theta)$, we denote by $S^\theta_{u_0}(f)\subset\mathcal V$ the solution set corresponding to $f$ of the problem (EHVI). That is, $u\in\mathcal W$ and there exists $\kappa\in L^1([0,T]\times\partial\Omega)$ such that $u(0)=u_0\in H$, $\kappa\in\partial j(u_N)=\widehat\theta(u_N)$ and
\[
\langle u'(t)+\mathscr A u(t)+ \mathscr B[u(t)],v\rangle+\int_{\partial{\Omega}} \kappa . v_N(z)\,d\sigma(z)=\la f(t),v\ra
\]for a.e  $t\in[0,T]$ and all $v\in V\cap L_N^\infty(\partial\Omega)$ .

\begin{lemma}
Let $f\in L^2(0,T;V^*)$. For every $u\in S^\theta_{u_0}(f)$, there exist $\delta_0,\, \delta_1>0$ such that  
\[
\int_{[0,T]\times\partial\Omega} \kappa .u_N\,d\sigma\,dt\geq -\delta_0\,\delta_1 T\sigma(\partial\Omega)
\]where $\delta_0$ and $\delta_1$ are independent from $u$.

\end{lemma}
\begin{proof}
By definition of $\underline\theta(u_N)$ and $\overline\theta(u_N)$ we have for every $\varepsilon>0$, there exists $\underline\delta$ with $|\mu|<\underline\delta$ such that 
\[
\underline{\theta}_\mu(u_N)-\varepsilon\leq \underline{\theta}(u_N)\leq \underline{\theta}_\mu(u_N)+\varepsilon
\]
and there exists $\overline\delta$ with $|\mu|<\overline\delta$ such that 
\[
\overline{\theta}_\mu(u_N)-\varepsilon\leq \overline{\theta}(u_N)\leq \overline{\theta}_\mu(u_N)+\varepsilon
\]
It follows that 
\[
\underline{\theta}_\mu(u_N)-\varepsilon\leq\kappa\leq \overline{\theta}_\mu(u_N)+\varepsilon 
\]
That is
\[
\displaystyle\essinf_{|u_N-s|<\mu}\theta(s)-\varepsilon\leq\kappa\leq  \displaystyle\esssup_{|u_N-s|<\mu}\theta(s)+\varepsilon
\]
Enlarging the bounds, we obtain 
\[
\displaystyle\essinf_{u_N-\mu\leq s<+\infty}\theta(s)-\varepsilon\leq \kappa\leq\displaystyle\esssup_{-\infty<s\leq u_N+\mu}\theta(s)+\varepsilon
\]
For $\varepsilon$ small enough and $|\mu|<\underline\delta,\,\overline\delta$, one have
\[
\displaystyle\essinf_{u_N-\underline\delta\wedge\overline\delta\leq s<+\infty}\theta(s)-\varepsilon\leq \kappa\leq\displaystyle\esssup_{-\infty<s\leq u_N+\underline\delta\wedge\overline\delta}\theta(s)+\varepsilon
\]
Consequently
\[
\displaystyle\sup_{u_N\in]-\infty,-\delta_0-\underline\delta\wedge\overline\delta[}\kappa\leq \displaystyle\esssup_{]-\infty,-\delta_0[}\theta(s)+\varepsilon
\]and
\[
\displaystyle\inf_{u_N\in]\delta_0+\underline\delta\wedge\overline\delta,+\infty[}\kappa\geq \displaystyle\essinf_{]\delta_0,+\infty[}\theta(s)-\varepsilon
\]where $\delta_0$ is defined in $H(\theta)$. Thus from $H(\theta)$ we obtain
\[
\displaystyle\sup_{u_N\in]-\infty,-\delta_0-\underline\delta\wedge\overline\delta[}\kappa\leq \varepsilon
\]and
\[
\displaystyle\inf_{u_N\in]\delta_0+\underline\delta\wedge\overline\delta,+\infty[}\kappa\geq-\varepsilon
\]
It results that $\kappa\leq \varepsilon$ if $u_N<-\delta_0-\underline\delta\wedge\overline\delta<-\delta_0$ and $\kappa\geq - \varepsilon$ if $u_N>\delta_0+\underline\delta\wedge\overline\delta>\delta_0$. We let $\varepsilon\to 0^+$ to get that $\kappa\leq 0$ if $u_N<-\delta_0$ and $\kappa\geq 0$ if $u_N>\delta_0$. Consequently, as $\theta\in L^\infty_{\text{loc}}(\mathbb R)$ and $\kappa\in\widehat\theta(u_N)$, then in the case $|u_N|\leq\delta_0$ we have 
\[
\displaystyle\sup_{|u_N|\leq\delta_0}|\kappa|\leq\displaystyle\esssup_{|s|\leq \delta_0}|\theta(s)|:=\delta_1
\]
It follows
\begin{align*}
\int_{[0,T]\times\partial\Omega}\kappa.u_N\,d\sigma\,dt=&\int_{\{u_N<-\delta_0\}}\kappa.u_N\,d\sigma\,dt+\int_{\{u_N>\delta_0\}}\kappa.u_N\,d\sigma\,dt\\
&\qquad\qquad\qquad+\int_{\{|u_N|\leq\delta_0\}}\kappa.u_N\,d\sigma\,dt\\
\geq&\, 0 -\delta_0\int_{\{|u_N|\leq\delta_0
\}}|\kappa|\,d\sigma\,dt\\
\geq & -\delta_0\delta_1 T\sigma(\partial\Omega)
\end{align*}

\end{proof}

\begin{theorem}\label{fn}
Under $H(\theta)$ assume that $f_m$, $f\in L^2(0,T;V^*)$ such that $f_m\rightarrow f$ weakly in $\mv^*$. Let $\{u_m\}_m\subset\mathcal W$ be a sequence such that $u_m\in S^\theta_{u_0}(f_m)$ for each $m\in\mathbb N$, then we can find a subsequence (still denoted with the same symbol) such that $u_m \rightarrow u$ weakly in $\mathcal V$ and $u\in S^\theta_{u_0 }(f)$.

\end{theorem}
\begin{proof}
Let $f_m,\,f\in  \mv^*$ with $f_m\rightarrow f$ weakly in $\mv^*$. Let $\{u_m\}_m$ be a sequence such that $u_m\in S^\theta_{u_0}(f_m)$ for each $m\in\mathbb N$, then by Theorem \ref{existence}, there exists $\kappa_m\in L^1([0,T] \times\partial\Omega)$ such that $\kappa_m(t,x)\in\partial j(u_{mN}(t,x))$  for a.e $(t,x)\in[0,T]\times\partial\Omega$ and 
\begin{equation}\label{eqstab}
\langle u'_m(t)+\mathscr A u_m(t)+ \mathscr B[u_m(t)],v\rangle+\int_{\partial{\Omega}} \kappa_m\, v_N\,d\sigma=\la f_m(t),v\ra
\end{equation}for a.e $t\in[0,T]$ and all $v\in V\cap L^\infty_N(\partial\Omega)$.
With the same calculations as in the last section one obtains
\[
\frac{1}{2}\frac{d}{dt}|u_m(t)|^2+M\|u_m(t)\|_V^2\leq \frac{2}{M}\|f_m(t)\|_{V^*}^2+\frac{M}{2}\|u(t)\|_V^2-2\int_{\partial{\Omega}} \kappa_m\, u_{mN}\,d\sigma
\]
It follows
\[
\frac{d}{dt}|u_m(t)|^2+M\|u_m(t)\|_V^2\leq \frac{4}{M}\|f_m(t)\|_{V^*}^2 -2\int_{\partial{\Omega}} \kappa_m\, u_{mN}\,d\sigma
\]
Integrating over $(0,t)$ we get
\[
\abs{u_m(t)}^2+M\int_0^t\|u_m(s)\|_{V}^2\,ds\leq \frac{4}{M}\|f_m\|_{\mv^*}+\abs{u(0)}^2+2\delta_0\,\delta_1\,T\sigma(\partial\Omega)
\]
It follows that $\{u_m\}_m$ is bounded in $L^\infty(0,T;H)\cap L^2(0,T;V)$. Hence, by passing to a subsequence if necessary, there exists $u$ such that $\{u_m\}_m$ converges to $u$ weakly in $L^2(0,T;V)$ and weak$-^*$ in $L^\infty(0,T;V)$.
Using the compactness of the trace operator $\gamma$, we may assume that $\gamma\,u_m\rightarrow \gamma\, u$ in $L^2(0,T;H)$ and then $\gamma\,u_m(t,z)\rightarrow \gamma\, u(t,z)$ for a.e. $(t,z)\in [0,T]\times\partial\Omega$. Consequently, $u_{mN}(t,z)\rightarrow u_N(t,z)$ for a.e. $(t,z)\in[0,T]\times\partial\Omega$. Let us show that there exists $m_0\in\mathbb N$ such that 
  
\[
\partial j(u_{mN})\subset \partial j(u_N),\quad \text{for all }m\geq m_0
\]
As $u_{mN}\rightarrow u_N$ for a.e. $(t,z)\in[0,T]\times\partial\Omega$, we can find, by Egoroff's theorem, that for any $\alpha > 0$ we can determine $\omega\subset [0,T]\times \partial{\Omega}$   with  $(dt\times\sigma)(\omega) <\alpha$  such that
\begin{equation}
u_{mN}\rightarrow u_N,\quad\text{ uniformly on } [0,T]\times\partial{\Omega} \setminus \omega
\end{equation} 
with $u_N\in  L^{\infty}( ([0,T]\times\partial{\Omega}) \setminus \omega )$. Thus for any $\mu> 0$ there exists $m_0$ such that for all $m > m_0$ we have
\begin{equation*}
|u_{mN}- u_N|<\frac{\mu}{2},\qquad \text{a.e on } [0,T]\times \partial{\Omega}\setminus\omega
\end{equation*}
By using triangle inequality, we have that 

\begin{equation*}
\begin{array}{ll}
\overline{\theta}_{\frac{\mu}{2}}(u_{mN})&= \displaystyle\esssup_{|u_{mN}-  \xi|\leq \frac{\mu}{2}} \theta(\xi)\\
&\leq \displaystyle\esssup_{|u_N-  \xi|\leq  \mu} \theta(\xi)\\
&= \overline{\theta}_{\mu}(u_N)
\end{array}
\end{equation*}
Analogously we prove the inequality
\begin{equation*}
 \underline{\theta}_{\mu}(u_N)\leq \underline{\theta}_{\frac{\mu}{2}}(u_{mN})
\end{equation*}
Taking the limit $\mu\rightarrow 0^+$, we obtain for each $m\geq m_0$ 
\[
\kappa_m\in\widehat{\theta}(u_{mN})\subset \widehat{\theta}(u_N), \qquad \text{ a.e on } [0,T]\times \partial{\Omega}\setminus\omega
\]where $(dt\times\sigma)(\omega) < \alpha $. For  $\alpha $ as small as possible, we obtain
\[
\kappa_m\in\widehat{\theta}(u_{mN})\subset \widehat{\theta}(u_N), \qquad \text{ a.e on }[0,T]\times \partial{\Omega}
\]from which we can conclude follows $\{\kappa_m\}_m$ is bounded. By Dunford-Pettis theorem \cite[p.239]{Ek76}, we will show that the sequence $\{\kappa_m\}_m$ is weakly precompact in $L^1([0,T]\times\partial\Omega)$. For this end we show that for each $\mu>0$, there exists $\delta>0$ such that for $\omega\subset [0,T]\times\partial\Omega$, $(dt\times\sigma)(\omega)<\delta$
\[
\int_{\omega}|\kappa_m|\,d\sigma\,dt<\mu
\]
For some $a>0$ and remarking that in $\{|u_{mN}|>a\}$, $1< \frac{|u_{mN}|}{a}$, one have
\begin{align*}
\int_{\omega}|\kappa_m|\,d\sigma\,dt=&\int_{\omega}|\kappa_m|1_{\{|u_{mN}|>a\}}\,d\sigma\,dt+\int_{\omega}|\kappa_m|1_{\{|u_{mN}|\leq a\}}\, d\sigma\,dt\\
\leq&\frac{1}{a}\int_{[0,T]\times\partial\Omega}|\kappa_m\,u_{mN}|\,d\sigma\,dt+\int_{\omega}|\kappa_m|1_{\{|u_{mN}|\leq a\}}\,d\sigma\,dt
\end{align*}
From one hand one has
\begin{align*}
\int_{[0,T]\times\partial\Omega}|\kappa_m\,u_{mN}|\, \sigma\,dt=&\int_{\{|u_{mN}|>\delta_0\}}|\kappa_m\,u_{mN}|\,d\sigma\,dt+\int_{\{|u_{mN}|\leq\delta_0\}}|\kappa_m\,u_{mN}|\,d\sigma\,dt\\
=&\int_{\{|u_{mN}|>\delta_0\}}|\kappa_m\,u_{mN}|\,d\sigma\,dt-\int_{\{|u_{mN}|\leq\delta_0\}}|\kappa_m\,u_{mN}|\,d\,\sigma\\&\qquad\qquad+2\int_{\{|u_{mN}|\leq\delta_0\}}|\kappa_m\,u_{mN}|\,d\sigma\,dt\\
\leq&\int_{\{|u_{mN}|>\delta_0\}}|\kappa_m\,u_{mN}|\,d\sigma\,dt+\int_{\{|u_{mN}|\leq\delta_0\}}\kappa_m\,u_{mN}\,d\sigma\,dt\\
&\qquad\qquad+2\int_{\{|u_{mN}|\leq\delta_0\}}|\kappa_m\,u_{mN}|\,d\sigma\,dt\\
=& \int_{[0,T]\times\partial\Omega}\kappa_m\,u_{mN}\,d\sigma\,dt+2\int_{\{|u_{mN}|\leq\delta_0\}}|\kappa_m\,u_{mN}|\,d\sigma\,dt
\end{align*}
From the equation \eqref{eqstab} with $v=u_m(t)$
\begin{align*}
\int_{[0,T]\times\partial\Omega}|\kappa_m\,u_{mN}|\, \sigma\,dt \leq &\int_0^T\la f_m(t),u_m(t)\ra\,dt-\int_0^T\la u_m'(t),u_m(t)\ra\,dt\\
&-\int_0^T\la\mathscr A u_m(t),u_m(t)\ra\,dt+2\int_{\{|u_{mN}|\leq\delta_0\}}|\kappa_m\,u_{mN}|\,d\sigma\,dt 
\\
\leq& \int_0^T \|f_m(t)\|_{V^*}\|u_m(t)\|_V\,dt-\frac{1}{2}|u(T)|^2+\frac{1}{2}|u(0)|^2\\
&-M\int_0^T\|u_m(t)\|_V^2\,dt+2\delta_0\int_{\{|u_{mN}|\leq\delta_0\}}|\kappa_m|\,d\sigma\,dt\\
\leq&\frac{1}{2}\int_0^T\|f_m(t)\|^2_{V^*}\,dt+\frac{1}{2}\int_0^T\|u_m\|^2_{V}\,dt+\frac{1}{2}|u_0|^2
\\
&\qquad+2\delta_0\,(dt\times\sigma)(\{|u_{mN}|\leq\delta_0\})\,\displaystyle\sup_{|u_{mN}|\leq\delta_0}|\kappa_m|
\\
\leq& c+2\delta_0\delta_1 T\sigma(\partial\Omega)
\end{align*}
On the other hand, for each $\varepsilon>0$ there is $a_\varepsilon>0$, such that for $|\nu|<a_\varepsilon$ one have
\begin{align*}
|\kappa_m|\leq&\displaystyle\esssup_{|s-u_{mN}|<\nu}|\theta(s)|+\varepsilon\\
\leq&\displaystyle\esssup_{|s-u_{mN}|<a_\varepsilon}|\theta(s)|+\varepsilon
\end{align*}
This implies 
\[
\displaystyle\sup_{|u_{mN}|\leq a}|\kappa_m|\leq \displaystyle\esssup_{|s|<a+a_\varepsilon}|\theta(s)|+\varepsilon
\]
We choose for example $\varepsilon=1$, which leads to
\[
\displaystyle\sup_{|u_{mN}|\leq a}|\kappa_m|\leq \displaystyle\esssup_{|s|<a+a_1}|\theta(s)|+1
\]
Now we choose $a$ such that 
\[
\frac{1}{a}\int_{[0,T]\times\partial\Omega}|\kappa_m\,u_{mN}|\,d\sigma\,dt\leq \frac{1}{a}(c+2\delta_0\delta_1 T\sigma(\partial\Omega))<\frac{\mu}{2}
\]
and $\delta$ such that 
\[
\displaystyle\esssup_{|s|<a+a_1}|\theta(s)|+1<\frac{\mu}{2\delta}
\]
With this choice of $\delta$ one have
\begin{align*}
\int_{\omega}|\kappa_m|1_{\{|u_{mN}|\leq a\}}\,d\sigma\,dt\leq&\, \displaystyle\sup_{|u_{mN}|\leq a}|\kappa_m|\,(dt\times\sigma)(\omega)\\
\leq & \,(\displaystyle\esssup_{|s|<a+a_1}|\theta(s)|+1)\,(dt\times\sigma)(\omega)\\
<& \frac{\mu}{2\delta}\delta=\frac{\mu}{2}
\end{align*}
It follows
\begin{align*}
\int_{\omega}|\kappa_m|d\,\sigma\leq&\frac{1}{a}\int_{\partial\Omega}|\kappa_m\,u_{mN}|d\,\sigma+\int_{\omega}|\kappa_m|1_{\{|u_{mN}|\leq a\}}d\,\sigma\\
<&\frac{\mu}{2}+\frac{\mu}{2}=\mu
\end{align*}
Consequently, we can extract from $\{\kappa_m\}_m$ a subsequence( denoted with the same symbol) that converges in $L^1((0,T)\times\partial\Omega)$ to some $\kappa\in L^1((0,T)\times\partial\Omega)$. By passing to the limit in \eqref{eqstab}, we get  
\begin{equation*}
\langle u'(t)+\mathscr A u(t)+ \mathscr B[u(t)],v\rangle+\int_{\partial{\Omega}} \kappa\, v_N\,d\sigma=\la f(t),v\ra
\end{equation*}with 
\[
\kappa\in \overline{\mathrm{conv}}\,\, \widehat{\theta}(u_N)=\widehat{\theta}(u_N),\quad\,\text{a.e on } [0,T]\times \partial{\Omega}
\]

\end{proof}
\begin{remark}
We will need Theorem \ref{fn} just for external forces in $L^2(0,T;H)$. As in this situation the duality between $V$ and $V^*$ coincides with the one on $H$, this will bring no more difficulties. 
\end{remark}
\begin{remark}
One can prove in the same way as in \cite[Theorem 5.1]{MBA} that the solutions of (EHVI) are stable under perturbation of $\theta$.

\end{remark}

In the remaining of this section, we will use the notation $S(f)$ instead of $S_{u_0}^\theta(f)$. We follow Migorski \cite{law2013note} and we let $\mathcal U=L^2(0,T;H)$ be the space of controls and $\mathcal U_{ad}$  a nonempty subset of $\mathcal U$ consisting of admissible controls. Let $\mathscr F:\mathcal U\times\mathcal V\rightarrow \mathbb R$ be the objective functional we want to minimize. The control problem reads as follows: 
\begin{equation}\label{control}
 \left\{ 
\begin{array}{ll}
\text{ Find a control }\hat f\in\mathcal U_{ad} \text{ and a state } \hat u\in S(\hat f) \text{ such that}:\\ \mathscr F(\hat f,\hat u)=\inf\left\{\mathscr F(f,u):\,f\in\mathcal U_{ad},\, u\in S( f)  \right\}

\end{array}
\right.
\end{equation}
A pair $(\hat f,\hat u)$ which solves \eqref{control} is called an optimal solution. The existence of such optimal solutions can be proved by using Theorem \ref{fn}. To do so, we need the following additional hypotheses:

\begin{itemize}
\item[~]$H(\mathcal U_{ad})$ $\quad \mathcal U_{ad}$ is a bounded and weakly closed subset of $\mathcal U$.
\end{itemize}

\begin{itemize}
\item[~] $H(\mathscr F)$ $\quad \mathscr F$ is lower semicontinuous with respect to $\mathcal U\times\mathcal V$ endowed with the weak topology.
\end{itemize}

\begin{theorem}\label{optsol}
Assume that $H(\theta)$, $H(\mathcal U_{ad})$ and $H(\mathscr F)$ are fulfilled. Then the problem \eqref{control} has an optimal solution.
\end{theorem}

\begin{proof}
Let $(f_m,u_m)$ be a minimizing sequence for the problem \eqref{control}, i.e $f_m\in\mathcal U_{ad}$ and $u_m\in S(f_m)$ such that 
\[
\displaystyle\lim_{m\to\infty}\mathscr F(f_m,u_m)=\inf\left\{\mathscr F(f,u):\,f\in\mathcal U_{ad},\, u\in S( f) \right\}=:\vartheta
\]
It follows that the sequence $f_m$ belongs to a bounded subset of the reflexive Banach space $\mathcal V$. We may then assume that $f_m\rightarrow \hat f$ weakly in $\mathcal V$ (by passing to a subsequence if necessary). By $H(\mathcal U_{ad})$, we have $\hat f\in\mathcal U_{ad}$. From Theorem \ref{fn}, we obtain, by again passing to a subsequence if necessary, that $u_m\rightarrow \hat u$ weakly in $\mathcal V$ with $\hat u\in S(\hat f)$. By $H(\mathscr F)$, we have $\vartheta\leq\mathscr F(\hat f,\hat u)\leq \displaystyle\liminf_{m\to\infty}\mathscr F(f_m,u_m)=\vartheta$. Which completes the proof.
\end{proof}

Next we apply Theorem \ref{optsol} in a concrete example. Let $X$ be another Hilbert space, $\mathcal X:=L^2(0,T;X)$, $\mathcal X_{ad }\subset\mathcal X$ the set of admissible controls and $C\in \mathscr L(\mathcal X,\mathcal U)$ a bounded linear operator from $\mathcal X$ to $\mathcal U$. Let $f\in L^2(0,T; V^*)$, we aim to study the following optimal control problem
\begin{equation}\label{control2}
 \left\{ 
\begin{array}{ll}
\text{ Find a control }w\in\mathcal X_{ad} \text{ and a state } \hat u\in S(C\hat w+f) \text{ such that}:\\ \mathscr J(\hat w,\hat u)=\inf\left\{\mathscr J(w,u):\,w\in\mathcal X_{ad},\, u\in S( Cw+f)  \right\}

\end{array}
\right.
\end{equation}where the objective functional is given by
\[
\mathscr J(w,u)=\int_0^T\int_\Omega(u(t,x)-z(t,x))^2\,dx\,dt+\int_0^T h(w(t))\,dt
\]for some function $h:X\rightarrow\mathbb R$ and $z\in L^2(0,T;H)$. Such optimal control problems arise in a wide range of applications, particularly in fluids flow control. More specifically, one try to act on the flow in such a way that certain flow profile is stabilized or enforced by devices like actuators. Also sensors are used to provide necessary information for the actuation measured here by the control input operator $C$. Our goal is to minimize the discrepancy between the ideal velocity profile $z$ and the actual flow $u$. Moreover, the cost related to the actuators and the sensors  should be also minimized. A more sophisticated example of this framework is the blood flow in an artificial heart. The goal will be to avoid, among other things, the stagnation causing some serious hymodynamic problems.

Let us first announce the following corollaries of Theorem \ref{fn}:

\begin{corollary}\label{cor}
Under $H(\theta)$ assume that $\varphi_m$, $\varphi$, $f\in \mv^*$ such that $\varphi_m\rightarrow \varphi$ weakly in $\mv^*$. Then for every $u_m\in S(\varphi_m+f)$, we can find a subsequence (still denoted with the same symbol) such that $u_m \rightarrow u$ in $\mathcal V$ and $u\in S(\varphi+f)$.
\end{corollary}
\begin{proof}
It suffices to take $f_m=\varphi_m+f$ in Theorem \ref{fn}.
\end{proof}

\begin{corollary}\label{corc}
Under $H(\theta)$, assume that $f\in\mathcal V^*$ and $w_m$, $w\in \mathcal X$ are such that $w_m$ converges weakly to $w$ in $\mathcal X$. Then for every sequence $\{u_m\}_m$ such that $u_m\in S(Cw_m+f)$,  we can find a subsequence that converges weakly in $L^2(0,T;V)$ to $u\in S(Cw+f)$.
\end{corollary}
\begin{proof}
It suffices to take $\varphi_m=Cw_m$ in Corollary \ref{cor}.
\end{proof}

Assume the following:
\begin{itemize}

\item[~](i)\quad $f \in L^2(0,T,V^*)$ and $z\in L^2(0,T;H)$.

\item[~](ii) \quad $\mathcal X_{ad}$ is a weakly compact subset of $\mathcal X$.

\item[~](iii)\quad the function $h:X\rightarrow \mathbb R$ is convex, lower semi-continuous and satisfies the coercivity condition
\[
|h(w)|\geq \alpha|w|^2_X+\beta
\]for some $\alpha>0$ and $\beta\in \mathbb R$. $|.|_X$ stands for the norm of the Hilbert space $X$.
\end{itemize}

\begin{theorem}
If hypotheses $(\text i)-(\text{iii})$ and $H(\theta)$ hold, then the problem \eqref{control2} has an optimal solution.
\end{theorem}

\begin{proof}
Let $(w_n,u_n)$ be a minimizing sequence to the problem \eqref{control2}, i.e $w_n\in\mathcal X_{ad}$ and $u_m\in S(Cw_n+f)$ such that 
\[
\displaystyle\lim_{m\to\infty}\mathscr J(w_m,u_m)=\inf\left\{\mathscr J(w,u):\,w\in\mathcal X_{ad},\, u\in S(Cw+ f) \right\}
\]
Denote $f_m=Cw_m+f$, $\mathscr F(f_m,u_m)=\mathscr J(w_m,u_m)$ and $\mathcal U_{ad}=C\mathcal X_{ad}$. It suffices now to apply Theorem \ref{optsol} for $\mathcal U_{ad}$ and $\mathscr F$.
\end{proof}
\section{Directional growth condition}

As mentioned in the introduction, Rauch condition is a particular case of the directional growth condition due to Z. Naniwiecz \cite{nani94}. It is of common knowledge that the foregoing mentioned conditions are sufficient to establish the existence of solution without any additional growth hypothesis on $j$. The notion of being solution needs only to be modified. Here we will reconsider the same problem of evolutionary hemivariational Navier-Stokes equations but with the more general condition of directional growth.

Let $j:\partial\Omega\times\mathbb R\rightarrow \mathbb R$ be a measurable function with respect to the first argument and locally Lipschitz with respect to the second argument. We assume the following:
\begin{enumerate}
\item[~]$H(j)$\quad there exists $\beta:\partial\Omega\times \mathbb R^+\rightarrow \mathbb R$ integrable with respect to the first argument and nondecreasing with respect to the second argument such that 
\[
|j(x,\xi)-j(x,\eta)|\leq\beta(x,r)|\xi-\eta|,\quad\forall \xi,\,\eta\in B(0,r),\, r\geq 0
\]
\end{enumerate}
\begin{enumerate}
\item[~]$H(j^0)$\quad there exists a function $\alpha:\partial\Omega\times \mathbb R^+\rightarrow\mathbb R$ square-integrable with respect to the first argument and nondecreasing with respect to the second argument such that the following estimate holds
\[
j^0(x;\xi,\eta-\xi)\leq\alpha(x,r)(1+|\eta|)
\]for almost every $x\in\partial\Omega$ and for any $\xi,\,\eta\in\mathbb R$ with $-r\leq\eta\leq r$,\, $r\geq 0$.
\end{enumerate}
Remark that if $j$ don't depends on $x\in\partial\Omega$ then it satisfies $H(j)$ automatically. The hypothesis $H(j^0)$ is called the directional growth condition.

For $x\in\partial\Omega$ and $\xi\in\mathbb R$, define $j_\varepsilon(x,\xi)=\mathfrak h{_\varepsilon}\star j(x,.)(\xi)$ and denote by $j_\varepsilon'$ the derivative of $j_\varepsilon$ with respect to the second argument. As usual let $\{\varphi_1,\varphi_2,\dots\}$ be a basis in $V\cap L^\infty_N(\partial\Omega)$ and $V_m=\text{span}\{\varphi_1,\varphi_1,\dots,\varphi_m\}$. We then consider the following regularized problem of Galerkin type associated with $(EHVI)$, noted $(\mathscr P^m_\varepsilon)$: Find $u_m\in \mathcal W_m$ such that $u_m(0)=u_{0m}$ and
\[
\la u_m'(t)+\mathscr A u_m(t)+\mathscr B[u_m(t)],v\ra+\int_{\partial\Omega} j_{\varepsilon_m}'(u_{mN}).v_N\,d\sigma=\la f(t),v\ra
\]for all $t\in(0,T)$ and all $v\in V_m$.

Note that due to the integrability of $\beta$ with respect to the first argument and $H(j)$ , the integral above is finite for each $u_m,\,v\in V_m$. In fact, we have
\[
\left|j_{\varepsilon_m}'(x;u_{mN(x)}).v_N(x)\right|\leq \beta(x,\|u_N\|_{L^\infty(\partial\Omega)}+1)\|v_N\|_{L^\infty(\partial\Omega)}
\]
Since $\beta(.,\|u_N\|_{L^\infty(\partial\Omega)}+1)\in L^1(\partial\Omega)$ the integrability of $j_{\varepsilon_m}'(u_{mN}).v_N$ over $\partial\Omega$ follows immediately for any $v\in V_m$. We have the following lemma(cf. \cite[Lemma 3.1]{nani94})

\begin{lemma}\label{estj}
Suppose that $H(j^0)$ holds. Then the estimate
\begin{equation}
j_{\varepsilon}'(x;\xi)(\eta-\xi)\leq \bar\alpha(x,r)(1+|\xi|),\qquad 0<\varepsilon<1
\end{equation}is valid for any $\xi,\,\eta\in\mathbb R$ with $-r\leq \eta\leq r$, $r\geq 0$ and almost all $x\in\partial\Omega$ where $\bar\alpha(x,r)=2\alpha(x,r+1)$.
\end{lemma}

The problem $(\mathscr P_\varepsilon^m)$ has at least one solution in $V_m$. In fact, substitution of $u_m(t)=\sum_{k=1}^m c_{km}(t)\varphi_k$ gives an initial value problem for a system of first order ordinary differential equations for $c_{km}(.)$, $k=1,\,2,\,\dots,\,m$. Its solvability on some interval $[0,t_m)$ follows from Carath\'eodory theorem. This solution can be extended on the closed interval $[0,T]$ by using the a priori estimates below.

Using the coercivity of $\mathscr A$, properties of $\mathscr B[.]$ and Young inequality, we have
\begin{align*}
\frac{1}{2}\frac{d}{dt}|u_m(t)|_H+M\|u_m(t)\|^2_V+\int_{\partial\Omega} j_{\varepsilon_m}'(u_{mN}(t)).u_{mN}(t)\,d\sigma
\leq \frac{M}{2}\|u_m(t)\|^2_V+\frac{2}{M}\|f(t)\|^2_{V^*}
\end{align*}
From Lemma \ref{estj}, we have
\[
j_{\varepsilon_m}'(u_{mN}).u_{mN}\geq -\bar\alpha(x,0)(1+|u_{mN}|)
\]
Then
\begin{align*}
\frac{1}{2}\frac{d}{dt}|u_m(t)|_H+&M\|u_m(t)\|^2_V
\leq \frac{2}{M}\|f(t)\|^2_{V^*}+\int_{\partial\Omega} \bar\alpha(x,0)(1+|u_{mN}|)\,d\sigma\\
\leq &\frac{2}{M}\|f(t)\|^2_{V^*}+\sigma(\partial\Omega)^{\frac{1}{2}}\|\bar\alpha(0)\|_{L^2(\partial\Omega)}+\|\bar\alpha(0)\|_{L^2(\partial\Omega)}\|u_{mN}(t)\|_{L^2(\partial\Omega)}\\
\leq &\frac{2}{M}\|f(t)\|^2_{V^*}+\sigma(\partial\Omega)^{\frac{1}{2}}\|\bar\alpha(0)\|_{L^2(\partial\Omega)}+\|\bar\alpha(0)\|_{L^2(\partial\Omega)}\|\gamma\|\|u_{m}(t)\|_V
\end{align*}
Integrating over $(0,t)$ we get
\begin{align*}
\frac{1}{2}|u_m(t)|_H&-\frac{1}{2}|u_{0m}|_H^2+\frac{M}{2}\int_0^t\|u_m(s)\|_V^2\,ds\\
\leq &\frac{2}{M}\|f\|^2_{\mathcal V^*}+\sigma(\partial\Omega)^{\frac{1}{2}}\|\bar\alpha(0)\|_{L^2(\partial\Omega)} t+\|\bar\alpha(0)\|_{L^2(\partial\Omega)}\|\gamma\|\int_0^t\|u_{mN}(s)\|_V\,ds
\end{align*}
It follows

\begin{align*}
\frac{1}{2}|u_m(t)|_H+\frac{M}{2}\|u_m\|_{\mathcal V}^2\leq & \frac{1}{2}|u_{0m}|_H^2+\frac{2}{M}\|f\|^2_{\mathcal V^*}\\
+&\sigma(\partial\Omega)^{\frac{1}{2}}T\|\bar\alpha(0)\|_{L^2(\partial\Omega)} +\sqrt T\|\bar\alpha(0)\|_{L^2(\partial\Omega)}\|\gamma\|\|u_{m}\|_{\mathcal V}
\end{align*}
It follows that $\{u_m\}_m$ is a bounded subset of $\mathcal V$ and $L^\infty(0,T;H)$, so passing to subsequence, if necessary, we have that 
\[
u_m\rightarrow u\quad\text{weakly in }\mathcal V\quad\text{and weakly}-^*\text{ in }L^\infty(0,T;H)
\]
Following the same procedure as in section 4, see also\cite{M07HE}, we have $u_m'\rightarrow u'$ weakly in $\mathcal W$ and $\mathscr B[u_m]\rightarrow \mathscr B[u]$ weakly in $\mathcal V^*$. Using the same proof as in \cite[Lemma 3.3]{nani94}, one can prove that the sequence $\{j_{\varepsilon_m}'(u_{mN}\}$ is weakly precompact in $L^1((0,T)\times\partial\Omega)$. This means that
\[
j_{\varepsilon_m}'(u_{mN})\longrightarrow\kappa\quad \text{weakly in } L^1((0,T)\times\partial\Omega)
\]
Moreover, the following equality holds
\[
\la u_m'(t)+\mathscr A u_m(t)+\mathscr B[u_m(t)],v\ra+\int_{\partial\Omega} j_{\varepsilon_m}'(u_{mN}).v_N\,d\sigma=\la f(t),v\ra
\]for almost every $t\in(0,T)$ and $v\in V\cap L_N^\infty(\partial\Omega)$. We pass to the limit as usual to obtain
\[
\la u'(t)+\mathscr A u(t)+\mathscr B[u_n(t)],v\ra+\int_{\partial\Omega} \kappa(t).v_N\,d\sigma=\la f(t),v\ra
\]for almost every $t\in[0,T]$ and $v\in V\cap L_N^\infty(\partial\Omega)$.

It still to prove that $\kappa(t,x)\in\partial j(x,u_N(t,x))$ for almost every $t\in (0,T)$ and $x\in\partial\Omega$. Since $\gamma u_n\longrightarrow \gamma u$ in $L^2(0,T;L^2(\partial\Omega;\mathbb R^d))$, we obtain that $u_{mN}\longrightarrow u_N$ in $L^2(0,T;L^2(\partial\Omega))$  and consequently for almost every $t\in (0,T)$
\[
u_{mN}(t,x)\longrightarrow u_N(t,x)\quad \text{a.e. }x\in\partial\Omega
\]
By Egoroff's theorem, with respect to $x\in\partial\Omega$, we have for any $\rho>0$ a subset $\omega$ of $\Gamma$ with
\[
u_{mN}\longrightarrow u_N\quad\text{uniformly on }\partial\Omega\setminus\omega
\]with $u_N\in L^\infty(\partial\Omega\setminus\omega)$. Let $v\in L_N^\infty(\partial\Omega\setminus\omega)$ be arbitrary given. Due to the Fatou's lemma, for any positive $\varepsilon>0$ there exists $\delta_\varepsilon>0$ and $N_\varepsilon$ such that 
\[
\int_{\partial\Omega\setminus\omega}\frac{j(x,u_{mN}(x)-\tau+\nu v_N(x))-j(x,u_m(x)-\tau)}{\nu}\,d\sigma\leq \int_{\partial\Omega\setminus\omega} j^0(x,u_N(x);v_N(x))\,d\sigma+\varepsilon
\]provided $m>N_\varepsilon$, $|\tau|<\delta_\varepsilon$ and $0<\nu<\delta_\varepsilon$. This inequality multiplied by $\mathfrak h_{\varepsilon_m}$ and integrated over $\mathbb R$ yields
\begin{align*}
\int_{\partial\Omega\setminus\omega}&\frac{j_{\varepsilon_m}(x,u_{mN}(x)+\nu v_N(x))-j_{\varepsilon_m}(x,u_{mN}(x))}{\nu}\,d\sigma(x)\\=&\int_{\mathbb R} \mathfrak h_{\varepsilon_m}(\tau)\int_{\partial\Omega\setminus\omega}\frac{j(x,u_{mN}(x)-\tau+\nu v_N(x))-j(x,u_n(x)-\tau)}{\nu}\,d\sigma\,d\tau\\
\leq & \int_{\partial\Omega\setminus\omega} j^0(x,u_N(x);v_N(x))\,d\sigma+\varepsilon
\end{align*}
But as $\nu\rightarrow 0$ we get
\[
\int_{\partial\Omega\setminus\omega}j_{\varepsilon_m}'(x,u_N(x)).v_N(x)\,d\sigma\leq \int_{\partial\Omega\setminus\omega} j^0(x,u_N(x);v_N(x))\,d\sigma+\varepsilon
\]which is valid for $m>N_\varepsilon$. Now letting $m\rightarrow\infty$ we are led to
\[
\int_{\partial\Omega\setminus\omega}\kappa.v_N\,d\sigma \leq \int_{\partial\Omega\setminus\omega} j^0(x,u_N(x);v_N(x))\,d\sigma+\varepsilon
\]
Since $\varepsilon>0$ was chosen arbitrary

\[
\int_{\partial\Omega\setminus\omega}\kappa.v_N\,d\sigma \leq \int_{\partial\Omega\setminus\omega} j^0(x,u_N(x);v_N(x))\,d\sigma\quad\text{for all }v\in L_N^\infty(\partial\Omega\setminus\omega)
\]
But the last inequality easily implies that 
\[
\kappa(x)\in\partial j(x,u_N(x))\quad \text{for a.e }x\in\partial\Omega\setminus\omega
\]where $\sigma(\omega)<\rho$. Now since $\rho$ was chosen arbitrary 
\[
\kappa(x)\in\partial j(x,u_N(x))\quad \text{for a.e }x\in\partial\Omega
\]which completes the proof.

\begin{remark}
The directional growth condition is meant to study problems involving vector valued functions, i.e functions on $\mathbb R^N$. Our situation is simpler as $N=1$. In this case, the direction growth condition can be simplified to the following condition
\begin{equation}\label{G2}
j^0(x,\xi,-\xi)\leq k(x)|\xi|,\quad\forall \xi\in\mathbb R\text{ for a.e }x\in\partial\Omega
\end{equation}for some nonnegative function $k\in L^2(\partial\Omega)$ (cf. \cite[Remark 4.1]{nani94}). Moreover, the Rauch condition and sign condition fulfil also the estimate \eqref{G2}(cf. \cite[Remark 4.7]{nani94}).

\end{remark}

\begin{remark}
It is an easy task to check that the results in section 5, regarding optimal solution, are also valid if one replace the assumption $H(\theta)$ by the more general assumption $H(j^0)$.

\end{remark}

\section*{Acknowledgements}
We would like to express our gratitude to the Editor for taking time to handle the manuscript and to anonymous referees whose constructive comments are very helpful for improving the quality of our paper. We thank Prof. S. Migorski for pointing out that the Rauch and the growth conditions are completely independent.

\section*{Conflict of  interest}
The authors declare that they have no conflict of interest.

\end{document}